\pdfoutput=1
\RequirePackage{ifpdf}
\ifpdf 
\documentclass[pdftex]{sigma}
\else
\documentclass{sigma}
\fi

\numberwithin{equation}{section}

\newtheorem{Theorem}{Theorem}[section]
\newtheorem*{Theorem*}{Theorem}

 { \theoremstyle{definition}

\newtheorem{Remark}[Theorem]{Remark} }
\newcommand{\D}{{\rm d}}
\DeclareMathOperator{\I}{Im}
\DeclareMathOperator{\RE}{Re}
\DeclareMathOperator{\R}{Re}
\DeclareMathOperator{\Li}{Li}
\DeclareMathOperator{\lcm}{lcm}

\def\pd{\text{\textit\pounds}}

\DeclareMathOperator{\Span}{span}
\usepackage{enumitem}[2011/09/28]
\setenumerate{align=left, leftmargin=0pt,labelsep=.5em, labelindent=0\parindent,listparindent=\parindent,itemindent=*}

\usepackage{esint,colonequals,mathrsfs,extarrows}

\begin{document}
\allowdisplaybreaks

\renewcommand{\thefootnote}{}

\newcommand{\arXivNumber}{2403.16945}

\renewcommand{\PaperNumber}{079}
	
\FirstPageHeading

\ShortArticleName{New Evaluations of Inverse Binomial Series via Cyclotomic Multiple Zeta Values}

\ArticleName{New Evaluations of Inverse Binomial Series\\ via Cyclotomic Multiple Zeta Values\footnote{This paper is a~contribution to the Special Issue on Asymptotics and Applications of Special Functions in Memory of Richard Paris. The~full collection is available at \href{https://www.emis.de/journals/SIGMA/Paris.html}{https://www.emis.de/journals/SIGMA/Paris.html}}}

\Author{John M. CAMPBELL~$^{\rm a}$, M. Lawrence GLASSER~$^{\rm b}$ and Yajun ZHOU~$^{\rm cd}$}

\AuthorNameForHeading{J.M.~Campbell, M.L.~Glasser and Y.~Zhou}

\Address{$^{\rm a)}$ Department of Mathematics and Statistics, Dalhousie University,\\
\hphantom{$^{\rm a)}$}~Halifax, NS, B3H 4R2, Canada}
\EmailD{\href{mailto:jmaxwellcampbell@gmail.com}{jmaxwellcampbell@gmail.com}}

\Address{$^{\rm b)}$ Department of Physics, Clarkson University, Potsdam NY 13699, USA}
\EmailD{\href{mailto:laryg@clarkson.edu}{laryg@clarkson.edu}}

\Address{$^{\rm c)}$ Program in Applied and Computational Mathematics (PACM), Princeton University,\\
\hphantom{$^{\rm c)}$}~Princeton, NJ 08544, USA}
 \EmailD{\href{mailto:yajunz@math.princeton.edu}{yajunz@math.princeton.edu}}

 \Address{$^{\rm d)}$ {Academy of Advanced Interdisciplinary Studies (AAIS), Peking University,\\
\hphantom{$^{\rm d)}$}~Beijing 100871, P.R. China}}
  \EmailD{\href{mailto:yajun.zhou.1982@pku.edu.cn}{yajun.zhou.1982@pku.edu.cn}}

\ArticleDates{Received March 26, 2024, in final form August 28, 2024; Published online September 03, 2024}

\Abstract{Through the application of an evaluation technique based on cyclotomic multiple zeta values recently due to Au, we solve open problems on inverse binomial series that were included in a 2010 analysis textbook by Chen.}

\Classification{33B30; 11B65; 11M32}

\renewcommand{\thefootnote}{\arabic{footnote}}
\setcounter{footnote}{0}

\section{Introduction}
 In Hongwei Chen's 2010 textbook on classical analysis \cite{Chen2010}, the quest for provable closed forms of the following hypergeometric series was highlighted as an open problem \cite[p.~215]{Chen2010}:
\begin{align}
 & \sum_{n=0}^{\infty} \frac{1}{ (2n+1)^3 \binom{2n}{n}}, \label{firstopen} \\
 & \sum_{n=0}^{\infty} \frac{(-1)^{n}}{(2 n + 1)^{3} \binom{2n}{n}}. \label{secondopen}
\end{align}
 Series involving inverted binomial coefficients, as in the hypergeometric series in \eqref{firstopen}--\eqref{secondopen}, are ubiquitous in many areas of
 analysis, and the problem of evaluating such series is important in computational and experimental mathematics, with a particular regard toward the classic
 text \emph{Experimentation in Mathematics} \cite[Section~1.7]{BorweinBaileyGirgensohn2004}. For inverse binomial series involving higher powers as factors in
 the denominator, beyond linear or quadratic factors, the evaluation of such series is of a recalcitrant nature, even for the cubic case. This is
 evidenced by the rich history associated with the problem of proving the Chudnovsky brothers' formula \cite{ChudnovskyChudnovsky1998}
\begin{equation}\label{Chudnovskymotivating}
 \sum_{n=1}^{\infty} \frac{1}{ n^3 \binom{3n}{n} 2^{n} }
 = \pi G - \frac{33 \zeta (3)}{16}+\frac{\log ^32}{6}-\frac{\pi ^2 \log 2}{24},
\end{equation}
 as described in \cite{CampbellLevrietoappear}, where $ G\colonequals \sum_{n=0}^\infty\frac{(-1)^n}{(2n+1)^2}$ is Catalan's constant and $ \zeta(3)
 \colonequals \sum_{n=1}^\infty\frac{1}{n^3}$ is Ap\'ery's constant. In addition to the work of the Chudnovsky brothers
 \cite{ChudnovskyChudnovsky1998}, notable research contributions on inverse binomial series with negative powers as
 in \eqref{firstopen}--\eqref{Chudnovskymotivating} include
 \cite{Ablinger2017,ABRS2014,BorweinBroadhurstKamnitzer2001,Chu2021,DavydychevKalmykov2001,DavydychevKalmykov2004,
KalmykovVeretin2000,Kalmykov2007,Lehmer1985,SunZhou2024sum3k4k,vanderPoorten1980,WangXu2021,Weinzierl2004bn,
Zhou2022mkMpl,Zhou2023SunCMZV,Zucker1985}. These past references further motivate us to investigate the inverse binomial series
 \eqref{firstopen} and \eqref{secondopen}. In this work, we solve Hongwei Chen's open problems indicated above \cite[p.\ 215]{Chen2010},
 using a recent evaluation technique due to Au based on \emph{cyclotomic multiple zeta values} (CMZVs) and \textit{multiple polylogarithms} (MPLs).

The rest of this article is organized as follows. In Section~\ref{sec:CMZV}, we give a gentle introduction to CMZVs and MPLs, which
 provides background on the algorithms associated with our proofs of
 Theorems \ref{thm:Chen1}--\ref{thm:ChenCMZV}. In Section~\ref{sec:Chen}, we present computer-assisted proofs for the closed-form evaluations of
 \eqref{firstopen} and \eqref{secondopen}, while handling integral representations of these series by Au's {\tt MultipleZetaValues} package \cite{Au2022a}. In Section~\ref{sec:genChen}, we perform an in-depth analysis of CMZVs and MPLs, accommodating to convergent series in the form of \begin{align}
\mathscr{S}_{k}(z)\colonequals \sum_{n=0}^{\infty} \frac{z^{n}}{ (2n+1)^k \binom{2n}{n}}\label{eq:genChenSk}
\end{align} for suitable positive integers $k$ and complex numbers $z$, which generalize \eqref{firstopen} and \eqref{secondopen}.
 In~addition to Au's software \cite{Au2022a}, Panzer's {\tt HyperInt} package \cite{Panzer2015} will also be
 essential to our manipulations of MPLs related to these generalizations of Chen's series.

References \cite{ABRS2014} and \cite{DavydychevKalmykov2001} provide the computational methods
 applied in this paper. These applications are based on Mellin transform representations
 of factors appearing within summands, and these factors can be expressed
 with a single Mellin transform.
 The desired sum over non-negative integers $n $ is then given by the integral of
\[ \sum_{n=0}^{\infty} x^{n} z^{n} = \frac{1}{1-xz}, \]
 for a possibly subsidiary parameter $z$.
 This integral is then used to obtain evaluations for the desired sum, by setting $z \to 1$.

 \section[Cyclotomic multiple zeta values (CMZVs) and multiple polylogarithms (MPLs)]{Cyclotomic multiple zeta values (CMZVs)\\ and multiple polylogarithms (MPLs)} \label{sec:CMZV}

Let $ \mathbb Z_{>0}\colonequals \{1,2,3,\dots\}$ be the set of positive integers. A convergent series of the form
\begin{align}
 \Li_{s_{1}, \ldots, s_{m}}(z_1,\dots, z_m )
 \colonequals \sum_{n _{1}> \cdots > n_{m} \geq 1} \frac{z_{1}^{n_1}\cdots z_m^{n_m}}{n_{1}^{s_{1}} \cdots n_{m}^{s_{m}} }\label{eq:CMZV_Li_defn}
\end{align}is referred to as a \textit{cyclotomic multiple zeta value} (CMZV) of weight $ k\in\mathbb Z_{>0} $ and level $N\in\mathbb Z_{>0}$ if $ s_1,\dots,s_m\in\mathbb Z_{>0}$, $s_1+\dots+s_m =k$, and $ z_1^N=\dots=z_m^N=1$. At level $ N=1$, CMZVs are reduced to the \textit{multiple zeta values} (MZVs)

 \[ \zeta(s_{1}, \ldots, s_{m})
 \colonequals \sum_{n _{1}> \cdots > n_{m} \geq 1} \frac{1}{n_{1}^{s_{1}} \cdots n_{m}^{s_{m}} } ,\]
 which play important roles within experimental mathematics,
 as highlighted in the classic text by
 Borwein, Bailey, and Girgensohn \cite[Section~3]{BorweinBaileyGirgensohn2004}. CMZVs at levels $ N\in\{1,2,3,4,6\}$ feature prominently in the perturbative expansions of Feynman diagrams in quantum field theory
 \mbox{\cite{Ablinger2017,ABRS2014,DavydychevKalmykov2001,DavydychevKalmykov2004,KalmykovVeretin2000,Kalmykov2007,Laporta:2017okg,
LaportaRemiddi1996,Schnetz2018,Weinzierl2004bn}}.
The algorithmic structures of MZVs (together with some generalizations) have been elucidated by Brown \cite{Brown2009arXiv,Brown2009a,Brown2009b} and implemented by Panzer in the~\texttt{HyperInt} package \cite{Panzer2015}.

Collections of CMZVs defined in \eqref{eq:CMZV_Li_defn} with the same weight $ k\in\mathbb Z_{>0} $ and level $N\in\mathbb Z_{>0}$ span a $\mathbb Q$-vector space\begin{align}
\mathfrak Z_k(N)\colonequals \Span_{\mathbb Q}\left\{\Li_{s_{1}, \ldots, s_{m}}(z_1,\dots, z_m )
 \left|\,\begin{matrix}s_1,\dots,s_m\in\mathbb Z_{>0},\\z_{1}^{N}=\dots=z_m^N=1,\\(s_1,z_1)\neq(1,1),\\ \sum\limits_{j=1}^{m}s_{j}=k\end{matrix}\right. \right\}.
\label{eq:Zk(N)_defn}
\end{align}We retroactively set $ \mathfrak Z_0(N)\colonequals \mathbb Q$ for all $ N\in\mathbb Z_{>0}$.
The $ \mathbb Q$-vector spaces enjoy a filtration property $ \mathfrak Z_j(N)\mathfrak
Z_{ k}(N)\subseteq \mathfrak
Z_{j+ k}(N)$ for $ j,k\in \mathbb Z_{\geq0}$ and any fixed level $N$ \cite[Section~1.2]{Goncharov1998}, namely, whenever we have two numbers $ z_j\in \mathfrak Z_j(N)$ and $z_k\in\mathfrak Z_k(N) $, their product $z_jz_k $ is in $\mathfrak Z_{j+k}(N) $.

For small weights $k$ and levels $ N\in\{1,2,3,4,5,6,7,8,9,10,12\}$, Au's {\tt MultipleZetaValues} package \cite{Au2022a} allows us to express every member of $ \mathfrak Z_k(N)$ as a $\mathbb Q $-linear combination of the numbers in a spanning set,\footnote{Putatively, the spanning set produced by the command {\tt MZBasis[M,k]}
 in Au's {\tt MultipleZetaValues} package~\cite{Au2022a} is indeed a $ \mathbb Q$-vector basis for $ \mathfrak Z_k(M)$, but there is no definitive evidence for such claims beyond the cases of~$\mathfrak Z_1(M)$, $\mathfrak Z_2(1)$, and~$ \mathfrak Z_2(2)$.} extending the support of $ N\in\{1,2,4\}$
 cases in Panzer's \texttt{HyperInt} package~\cite{Panzer2015}. For example, we have
\begin{align*}
\mathfrak Z_2(4)=\Span_{\mathbb Q}\{{\rm i}G,\pi^{2},\pi {\rm i}\log 2,\log^22\}
\end{align*}involving Catalan's constant $G \colonequals \I\Li_{2}({\rm i})$.

One may also consider convergent series in the form of \eqref{eq:CMZV_Li_defn} without imposing the cyclotomic constraint that $ z_1^N=\dots=z_m^N=1$. This defines a \textit{multiple polylogarithm} (MPL) $ \Li_{s_{1}, \ldots, s_{m}}(z_1,\dots, z_m )
 $ of weight $ k=s_1+\dots+s_m$.
As a special case of MPLs, we have the \emph{polylogarithm} function
 $\Li_{s}(z) \colonequals \sum_{n=1}^{\infty} \frac{z^{n}}{n^s}$ of weight $ s\in\mathbb Z_{>0}$. Panzer's \texttt{HyperInt} package \cite{Panzer2015} allows one to reduce certain expressions involving MPLs, via Brown's algorithm \cite{Brown2009arXiv,Brown2009a,Brown2009b}.

 In particular, the polylogarithm of weight $3$ leads us to a Catalan-like constant \cite{CampbellChen2022,CampbellLevrieXuZhao2024}
\begin{equation}\label{eq:Catalanlike}
 \mathcal{G} \colonequals \I \Li_{3}\biggl( \frac{{\rm i}+1}{2} \biggr)=-\I \Li_{2,1}({\rm i},1)-\frac{G \log 2}{2} +\frac{\pi \log ^22}{32}+\frac{3 \pi ^3}{128} ,
\end{equation} which becomes useful in the construction of a spanning set for $\mathfrak Z_3(4) $, namely
\begin{align*}
\mathfrak Z_3(4)=\Span_{\mathbb Q}\bigl\{\zeta (3),{\rm i}\mathcal{G} ,\pi G,{\rm i} G \log 2,{\rm i} \pi ^3,\pi ^2 \log 2,{\rm i} \pi \log ^22,\log ^32\bigr\}.
\end{align*}
Here, $ \zeta(3)\colonequals \Li_3(1)$ is Ap\'ery's constant.

From now on, we will follow the practices of \cite{SunZhou2024sum3k4k,Zhou2022mkMpl,Zhou2023SunCMZV}, where
 special values of natural logarithms are abbreviated as follows:
\begin{gather}
\lambda\colonequals  \log2=-\Li_1(-1)\in\mathfrak Z_1(2),\nonumber\\
\varLambda\colonequals  \log 3= -2\RE\Li_1\bigl({\rm e}^{2\pi {\rm i}/3}\bigr)\in\mathfrak Z_1(3), \nonumber\\
\pd\colonequals  \log\frac{1+\sqrt{5}}{2}=\RE \bigl[\Li_1\bigl({\rm e}^{{2 \pi {\rm i}}/{5}}\bigr)- \Li_1\bigl({\rm e}^{{4 \pi {\rm i}}/{5}}\bigr)\bigr]\in \mathfrak Z_1(5), \nonumber\\
\mathscr L\colonequals  \log5=-2\RE \bigl[\Li_1\bigl({\rm e}^{{2 \pi {\rm i}}/{5}}\bigr)+ \Li_1\bigl({\rm e}^{{4 \pi {\rm i}}/{5}}\bigr)\bigr]\in \mathfrak Z_1(5), \nonumber\\
\widetilde\lambda\colonequals  \log\bigl(1+\sqrt{2}\bigr)=\RE\bigl[\Li_1\bigl( {\rm e}^{\pi {\rm i}/4}\bigr)-\Li_1\bigl( {\rm e}^{3\pi {\rm i}/4}\bigr)\bigr]\in \mathfrak Z_1(8),\nonumber\\
\widetilde\varLambda\colonequals  \log\bigl(2+\sqrt{3}\bigr)=2\RE\Li_1\bigl({\rm e}^{\pi {\rm i}/6}\bigr)\in \mathfrak Z_1(12).
\label{eq:log_abbrv}
\end{gather}
We bear in mind that products of these listed logarithms are CMZVs of higher weights, such as
\begin{align*}
\lambda\varLambda\in\mathfrak Z_2(6),\qquad \lambda\pd^2\in\mathfrak Z_3(10),
\end{align*}
by virtue of the natural embedding $ \mathfrak Z_k(N)\subseteq\mathfrak Z_k(M)$ for $N\mid M $, together with Goncharov's filtration $ \mathfrak Z_j(M)\mathfrak
Z_{ k}(M)\subseteq \mathfrak
Z_{j+ k}(M)$ \cite[Section~1.2]{Goncharov1998}.

 With the preparations so far, we can state the next two theorems to be proved in Section~\ref{sec:Chen}.

\begin{Theorem}\label{thm:Chen1}
 Recall $ \mathfrak Z_k(N)$ from \eqref{eq:Zk(N)_defn}, $ \mathcal G$ from \eqref{eq:Catalanlike}, and the abbreviations for special logarithms from \eqref{eq:log_abbrv}. Chen's series \[\sum_{n=0}^{\infty} \frac{1}{ (2n+1)^3 \binom{2n}{n}}\] admits the evaluation
\begin{align}
 & \frac{32\mathcal{G}}{3} - \frac{4\pi \Li_2\bigl(2-\sqrt{3}\bigr)}{3} -\frac{\pi^3}{9} -
 \frac{\pi \bigl(\lambda-\widetilde\varLambda\bigr)^2}{3} ,\label{eq:Chen1sln} \end{align}which belongs to the $\mathbb Q$-vector
 space $ {\rm i}\mathfrak Z_3(12)$.
\end{Theorem}

\begin{Theorem}\label{thm:Chen2}
 Set $\phi$ as the golden ratio
 $\frac{\sqrt{5} + 1}{2}$, so that $\pd=\log\phi $. Chen's series \[\sum_{n=0}^{\infty} \frac{(-1)^{n}}{ (2n+1)^3 \binom{2n}{n}}\] evaluates to
\begin{align}
&{} -\frac{4 \Li_3\bigl(\frac{1}{\phi^3}\bigr)}{3}-4 \Li_2 \biggl(\frac{1}{\phi^3}\biggr) \pd+\Li_3\biggl(\frac{1}{\phi}\biggr)-\frac{25\pd ^3}{3} +6 \lambda\pd ^2 + \frac{\pi ^2 \pd}{10} +\frac{12 \zeta (3)}{5}-\frac{\pi ^2 \lambda}{3} ,\label{eq:Chen2sln}
\end{align}
a number that lies in $ \mathfrak Z_3(10)$.
\end{Theorem}

In Section~\ref{sec:genChen}, we will unify Theorems \ref{thm:Chen1} and \ref{thm:Chen2} into a form given below.
\begin{Theorem}\label{thm:Li2Li3}
If $ |w|\leq1$, $ \R w>0$, $ \I w\geq0$, and $ \bigl|1-w^2\bigr|\leq2|w|$, then we have
\begin{align}
\sum _{n=0}^{\infty } \frac{(-1)^n }{(2 n+1)^3 \binom{2 n}{n}}\biggl(\frac{1-w^2}{w}\biggr)^{2 n+1}&{}=-2\biggl[\Li_3\biggl( \frac{1+w}{2} \biggr)-\Li_3\biggl( \frac{1-w}{2} \biggr)-\Li_3\biggl( \frac{1+\frac{1}{w}}{2} \biggr)\nonumber\\
&\quad{}+\Li_3\biggl( \frac{1-\frac{1}{w}}{2} \biggr) \biggr]+\biggl[\Li_2\biggl( \frac{1+w}{2} \biggr)-\Li_2\biggl( \frac{1-w}{2} \biggr)\nonumber\\
&\quad{}+\Li_2\biggl( \frac{1+\frac{1}{w}}{2} \biggr)-\Li_2\biggl( \frac{1-\frac{1}{w}}{2} \biggr) \biggr]\log w\nonumber\\
&\quad{}+\pi {\rm i}\log\biggl(\frac{1+w}{2}\biggr)\log\biggl(\frac{1+\frac{1}{w}}{2}\biggr),
\label{eq:genChenLi2Li3}
\end{align}
where the dilogarithm $ \Li_2(z)$ and the trilogarithm $ \Li_3(z)$ are defined by $($analytic continuations of$)$ $ \Li_s(z)\colonequals \sum_{n=1}^\infty\frac{z^n}{n^s}$ for $|z|\leq 1 $ and $s>1 $.
\end{Theorem}
In Section~\ref{sec:genChen}, we will also reveal the CMZV structures for further generalizations of Chen's series, as stated in the theorem below.
\begin{Theorem}\label{thm:ChenCMZV}
Denote the least common multiple of two numbers $a$ and $b$ by $ \lcm(a,b)$, and recall the definition of $ \mathscr{S}_k(z)$ from~\eqref{eq:genChenSk}.
\begin{enumerate}[leftmargin=*, label=\emph{(\alph*)},ref=(\alph*),widest=d, align=left]\itemsep=0pt
\item
For $ k-1,N-2\in\mathbb Z_{>0}$ and $ m\in\mathbb Z$, we have\begin{align}
\mathscr{S}_{k}\bigg(4\sin^2\frac{2m\pi}{N}\bigg)\sin\frac{2m\pi}{N}\in{}&{\rm i}\mathfrak Z_k(\lcm(2,N)).\label{eq:genChenZa}
\end{align}
\item For $ k-1\in\mathbb Z_{>0}$, the following relations hold true:
\begin{align}
&\mathscr{S}_{k}\biggl( -\frac{9}{4} \biggr)\in\mathfrak Z_k(6),\label{eq:genChen'a}\\
&\mathscr{S}_{k}(-4)\in\mathfrak Z_k(8),\label{eq:genChen'b}\\
&\sqrt{2}\mathscr{S}_{k}\biggl(-\frac12\biggr)\in\mathfrak Z_k(8),\label{eq:genChen'b'}\\
&\mathscr{S}_{k}(-1)\in\mathfrak Z_k(10),\label{eq:genChen'c}\\
&\sqrt{5}\mathscr{S}_k\biggl( -\frac{16}{5} \biggr)\in\mathfrak Z_k(10),\label{eq:genChen'c'}\\
&\sqrt{3}\mathscr{S}_k\biggl( -\frac{4}{3} \biggr)\in\mathfrak Z_k(12).\label{eq:genChen'd}
\end{align}
\end{enumerate}
\end{Theorem}
Our proof of the last theorem will be both constructive and algorithmic. In particular, Au's {\tt MultipleZetaValues} package \cite{Au2022a} will provide us with many concrete CMZV characterizations of the infinite series $\mathscr{S}_k(z) $, such as (cf.\ \eqref{eq:Catalanlike} and \eqref{eq:log_abbrv} for notations)
\begin{gather}
\sqrt2 \mathscr{S}_{3}(2)=  -8\I\Li_3\biggl(\frac{1-{\rm e}^{\pi {\rm i}/4}}{2} \biggr) -4\I\Li_3\bigl({\rm i} \bigl(\sqrt{2}-1\bigr)\bigr)\nonumber\\ \hphantom{\sqrt2 \mathscr{S}_{3}(2)=}{}
-\frac{\pi \bigl[48 \Li_2\bigl(\sqrt{2}-1\bigr)-12 \lambda \widetilde{\lambda }+20 \widetilde{\lambda }^2+9 \lambda ^2\bigr]}{32}+\frac{15 \pi ^3}{128}\in {\rm i}\mathfrak Z_3(8),
 \label{eq:S3(2)}\\
\sqrt{2}\mathscr{S}_4(2)= -36 \I\Li_4\bigl(1-{\rm e}^{\pi {\rm i}/4}\bigr)-12 \I\Li_4\biggl(\frac{1-{\rm e}^{\pi {\rm i}/4}}{2}\biggr)-12 \I\Li_4\bigl({\rm i}\bigl(1-{\rm e}^{\pi {\rm i}/4}\bigr)\bigr)
\nonumber\\ \hphantom{\sqrt{2}\mathscr{S}_4(2)=}{}
-12 \I\Li_4\bigl({\rm i} \bigl(\sqrt{2}-1\bigr)\bigr)
-\frac{9\beta(4)}{2}-14 \sqrt{2} L_{8,4}(4)+\frac{10\pi \sqrt{2} L_{8,2}(3)}{3} \nonumber\\ \hphantom{\sqrt{2}\mathscr{S}_4(2)=}{}
-\frac{9\pi \Li_3\bigl(\frac{1}{\sqrt{2}}\bigr)}{2}+\frac{63 \pi \zeta (3)}{128}+\frac{\pi \bigl(78 \lambda ^2 \widetilde{\lambda }-12 \lambda \widetilde{\lambda }^2-24 \widetilde{\lambda }^3+47 \lambda ^3\bigr)}{256} \nonumber\\ \hphantom{\sqrt{2}\mathscr{S}_4(2)=}{}
-\frac{3 \pi ^3 \bigl(141 \lambda -98 \widetilde{\lambda }\bigr)}{1024}\in {\rm i}\mathfrak Z_4(8), \label{eq:S4(2)}\\
\sqrt{3}\mathscr{S}_3(3)= -8 \I\Li_3\biggl(\frac{1-{\rm i}\sqrt{3}}{4}\biggr)-5 \I\Li_3\Biggl(\frac{1+\frac{{\rm i}}{\sqrt{3}}}{2}\Biggr)+\frac{\pi \Li_2\bigl(\frac{1}{4}\bigr)}{3}\nonumber\\ \hphantom{\sqrt{3}\mathscr{S}_3(3)=}{}
  +\frac{\pi \varLambda^2}{48}-\frac{7 \pi ^3}{432} \in {\rm i}\mathfrak Z_3(6), \label{eq:S3(3)}\\
\sqrt{3}\mathscr{S}_4(3)= 8 \I\Li_4\biggl(\frac{3+{\rm i} \sqrt{3}}{4}\biggr)-8 \I\Li_4\biggl(\frac{1-{\rm i} \sqrt{3}}{4}\biggr)\nonumber\\ \hphantom{\sqrt{3}\mathscr{S}_4(3)=}{}
-5 \I\Li_4\Biggl(\frac{1+\frac{{\rm i}}{\sqrt{3}}}{2}\Biggr)-\frac{45\sqrt{3} L_{3,2}(4)}{16} +\frac{\pi\bigl[\Li_3\bigl(\frac{1}{3}\bigr)+\Li_3\bigl(\frac{1}{4}\bigr) \bigr]}{3} -\frac{19 \pi \zeta (3)}{36}
\nonumber\\  \hphantom{\sqrt{3}\mathscr{S}_4(3)=}{}
+\frac{\pi \bigl(64 \lambda ^3-192 \lambda ^2 \varLambda +144 \lambda \varLambda ^2-41 \varLambda ^3\bigr)}{288}
+\frac{\pi ^3 (144 \lambda -41 \varLambda )}{864} \in {\rm i}\mathfrak Z_4(6), \label{eq:S4(3)}\\
\mathscr{S}_{3}(4)= 4\mathcal G-\frac{\pi \lambda ^2}{8} -\frac{\pi ^3}{32}\in {\rm i}\mathfrak Z_3(4),\label{eq:S3(4)}\\
\mathscr{S}_4(4)= 8\I \Li_4\biggl(\frac{1+{\rm i}}{2}\biggr)-4 \beta(4)+\frac{\pi \lambda^3}{24} +\frac{\pi ^3 \lambda}{32}\in {\rm i}\mathfrak Z_4(4), \label{eq:S4(4)}
 \end{gather}
 and
 \begin{gather}
\mathscr{S}_{3}\biggl(-\frac{9}{4}\biggr)= \frac{4\Li_3\bigl(\frac{1}{3}\bigr)}{3} +2 \Li_3\biggl(\frac{1}{4}\biggr)-\frac{5\zeta (3)}{9} +2 \Li_2\biggl(\frac{1}{4}\biggr) \lambda\nonumber\\  \hphantom{\mathscr{S}_{3}\biggl(-\frac{9}{4}\biggr)=}{}
+\frac{2\bigl(6 \lambda ^3- \varLambda ^3\bigr)}{9}-\frac{\pi ^2 (3 \lambda -2 \varLambda)}{9}\in \mathfrak Z_3(6), \label{eq:S3(-9/4)}\\
\mathscr{S}_4\biggl(-\frac{9}{4}\biggr)= \frac{80 \Li_4\bigl(\frac{1}{2}\bigr)}{9}-\frac{40 \Li_4\bigl(\frac{1}{3}\bigr)}{3}+8 \Li_4\biggl(\frac{2}{3}\biggr)+\frac{7 \Li_4\bigl(\frac{1}{4}\bigr)}{2} \nonumber\\  \hphantom{\mathscr{S}_4\biggl(-\frac{9}{4}\biggr)=}{}
+\frac{5 \Li_4\bigl(\frac{1}{9}\bigr)}{6}+4 \Li_3\biggl(\frac{1}{3}\biggr) \lambda+3 \Li_3\biggl(\frac{1}{4}\biggr) \lambda-\frac{50 \zeta (3) \lambda}{9}\nonumber\\ \hphantom{\mathscr{S}_4\biggl(-\frac{9}{4}\biggr)=}{}
-\frac{35 \lambda ^4-54 \lambda ^2 \varLambda ^2+54 \lambda \varLambda ^3-9 \varLambda ^4}{27}-\frac{\pi ^2 \lambda (11 \lambda -36 \varLambda )}{54}
-\frac{101 \pi ^4}{1620}\in \mathfrak Z_4(6),\!\!\! \label{eq:S4(-9/4)}\\
\mathscr{S}_{3}(-4)= -2 \Li_3\bigl(\sqrt{2}-1\bigr)+\frac{4\sqrt{2} L_{8,2}(3)}{3} +\frac{25 \zeta (3)}{16}\nonumber\\ \hphantom{\mathscr{S}_{3}(-4)= }{}
-2 \Li_2\bigl(\sqrt{2}-1\bigr) \widetilde{\lambda }-\frac{2 \widetilde{\lambda }^3}{3} +\frac{\lambda \widetilde{\lambda }^2}{2}-\frac{\pi ^2 \lambda }{8}\in \mathfrak Z_3(8),
 \label{eq:S3(-4)}\\
\mathscr{S}_4(-4)= \frac{40\Li_4\bigl(1-\frac{1}{\sqrt{2}}\bigr)}{7} +\frac{4 \Li_4\bigl(\sqrt{2}-1\bigr)}{21}+\frac{4 \Li_4\bigl(\frac{1}{\sqrt{2}}\bigr)}{7}
 -\frac{27 \Li_4\bigl(\frac{1}{2}\bigr)}{28}-\frac{59\Li_4\bigl(\bigl( \sqrt{2}-1\bigr)^{2}\bigr)}{14} \nonumber\\ \hphantom{\mathscr{S}_4(-4)=}{}
 +\frac{19\Li_4\bigl(\bigl( \sqrt{2}-1\bigr)^{4}\bigr)}{84}-\frac{2\Li_4\Bigl(\frac{1-\frac{1}{\sqrt{2}}}{2}\Bigr)}{21}
  +\frac{8\sqrt{2} L_{8,2}(3) \widetilde{\lambda }}{3} -4 \Li_3\biggl(\frac{1}{\sqrt{2}}\biggr) \widetilde{\lambda }
 \nonumber\\ \hphantom{\mathscr{S}_4(-4)=}{}
+\frac{7 \zeta (3) \widetilde{\lambda }}{16}
 +\frac{600 \lambda ^3 \widetilde{\lambda }+1224 \lambda ^2 \widetilde{\lambda }^2+96 \lambda \widetilde{\lambda }^3-752 \widetilde{\lambda }^4-177 \lambda ^4}{4032}\nonumber\\ \hphantom{\mathscr{S}_4(-4)=}{}
 -\frac{\pi ^2 \bigl(189 \lambda \widetilde{\lambda }-61 \widetilde{\lambda }^2-30 \lambda ^2\bigr)}{504}+\frac{11 \pi ^4}{7560}\in \mathfrak Z_4(8),
 \label{eq:S4(-4)}\\
\sqrt{2}\mathscr{S}_3\biggl( -\frac{1}{2} \biggr)= -80 \Li_3\biggl(\frac{1}{\sqrt{2}}\biggr)+64 \sqrt{2} L_{8,2}(3)+\frac{35 \zeta (3)}{4}-20 \Li_2\bigl(\sqrt{2}-1\bigr) \lambda\nonumber\\ \hphantom{\sqrt{2}\mathscr{S}_3\biggl( -\frac{1}{2} \biggr)=}{}
+10 \lambda ^2 \widetilde{\lambda }-10 \lambda \widetilde{\lambda }^2+\frac{5 \lambda ^3}{3}
-\frac{15 \pi ^2 \lambda }{4}\in \mathfrak Z_3(8),
 \label{eq:S3(-1/2)}\\
\sqrt{2}\mathscr{S}_4\biggl( -\frac{1}{2} \biggr)= \frac{1669 \Li_4\bigl(\frac{1}{2}\bigr)}{14}-\frac{2112\Li_4\bigl(1-\frac{1}{\sqrt{2}}\bigr)}{7} -\frac{704 \Li_4\bigl(\sqrt{2}-1\bigr)}{21}
+\frac{24 \Li_4\bigl(\frac{1}{\sqrt{2}}\bigr)}{7}
\nonumber\\ \hphantom{\sqrt{2}\mathscr{S}_4\biggl( -\frac{1}{2} \biggr)=}{}
+\frac{1510\Li_4\bigl(\bigl( \sqrt{2}-1\bigr)^{2}\bigr)}{7}
-\frac{475\Li_4\bigl(\bigl( \sqrt{2}-1\bigr)^{4}\bigr)}{42} +\frac{352\Li_4\Bigl(\frac{1-\frac{1}{\sqrt{2}}}{2}\Bigr)}{21}\nonumber\\  \hphantom{\sqrt{2}\mathscr{S}_4\biggl( -\frac{1}{2} \biggr)=}{}
-100 \Li_3\biggl(\frac{1}{\sqrt{2}}\biggr) \lambda+\frac{224\sqrt{2} L_{8,2}(3) \lambda}{3} +\frac{175 \zeta (3)\lambda }{16}\nonumber\\ \hphantom{\sqrt{2}\mathscr{S}_4\biggl( -\frac{1}{2} \biggr)=}{}
+\frac{2\bigl(99 \lambda ^3 \widetilde{\lambda }-297 \lambda ^2 \widetilde{\lambda }^2-132 \lambda \widetilde{\lambda }^3+299 \widetilde{\lambda }^4+309 \lambda ^4\bigr)}{63} \nonumber\\ \hphantom{\sqrt{2}\mathscr{S}_4\biggl( -\frac{1}{2} \biggr)=}{}
+\frac{\pi ^2 \bigl(1848 \lambda \widetilde{\lambda }-1336 \widetilde{\lambda }^2-2115 \lambda ^2\bigr)}{252} +\frac{397 \pi ^4}{3780}\in \mathfrak Z_4(8),
 \label{eq:S4(-1/2)}\\
\sqrt{5}\mathscr{S}_3\biggl( -\frac{16}{5} \biggr)= \frac{5 \Li_3\bigl(\frac{1}{5}\bigr)}{4}+\frac{27\Li_3\bigl(\frac{1}{\phi}\bigr)}{2} -10 \Li_3\biggl(\frac{1}{\sqrt{5}}\biggr)-\frac{27 \zeta (3)}{20}\nonumber\\ \hphantom{\sqrt{5}\mathscr{S}_3\biggl( -\frac{16}{5} \biggr)=}{}
+\frac{5\bigl[\Li_2\bigl(\frac{1}{5}\bigr)-4\Li_2\bigl(\frac{1}{\sqrt{5}}\bigr)\bigr]\mathscr L}{8} -\frac{9 \pd ^3 }{2} +\frac{27\pi ^2 \pd}{20}-\frac{5\pi ^2 \mathscr L}{16} \in \mathfrak Z_3(10),
 \label{eq:S3(-16/5)}\\
\sqrt{3}\mathscr{S}_3\biggl( -\frac{4}{3} \biggr)= -\frac{21 \Li_3\bigl(\frac{1}{3}\bigr)}{10} -\frac{7 \Li_3\bigl(\frac{1}{4}\bigr)}{40}-\Li_3\biggl(\frac{\sqrt{3}-1}{2}\biggr)+\frac{11\Li_3\bigl(1-\frac{\sqrt{3}}{2}\bigr)}{20} +\frac{9\Li_3\bigl(\frac{2-\sqrt{3}}{3}\bigr)}{5}
\nonumber\\ \hphantom{\sqrt{3}\mathscr{S}_3\biggl( -\frac{4}{3} \biggr)=}{}
-7 \Li_3\bigl(2-\sqrt{3}\bigr)
+\frac{24\Li_3\bigl(2 \sqrt{3}-3\bigr)}{5} +\frac{11\Li_3\bigl(3 \sqrt{3}-5\bigr)}{5} +\frac{3\sqrt{3} L_{12,4}(3)}{5}\nonumber\\ \hphantom{\sqrt{3}\mathscr{S}_3\biggl( -\frac{4}{3} \biggr)=}{}
+\frac{39 \zeta (3)}{10}+\frac{3\Li_2\bigl(\frac{1}{4}\bigr)\varLambda}{8} -3 \Li_2\biggl(\frac{\sqrt{3}-1}{2}\biggr) \varLambda-3 \Li_2\bigl(2-\sqrt{3}\bigr) \varLambda-\frac{17 \lambda ^2 \widetilde{\varLambda }}{80}
\nonumber\\ \hphantom{\sqrt{3}\mathscr{S}_3\biggl( -\frac{4}{3} \biggr)=}{}
+\frac{71 \lambda \widetilde{\varLambda }^2}{80}-\frac{3\lambda \varLambda \widetilde{\varLambda }}{4} -\frac{3 \varLambda ^2 \widetilde{\varLambda }}{20}
+\frac{21 \varLambda \widetilde{\varLambda }^2}{40}-\frac{209 \widetilde{\varLambda }^3}{240}+\frac{13 \lambda ^3}{80}+\frac{3 \lambda ^2 \varLambda }{8}+\frac{\varLambda ^3}{20}\nonumber\\ \hphantom{\sqrt{3}\mathscr{S}_3\biggl( -\frac{4}{3} \biggr)=}{}
+\frac{7 \pi ^2 \widetilde{\varLambda }}{40}-\frac{7 \pi ^2 \lambda }{20}-\frac{7 \pi ^2 \varLambda }{80}\in\mathfrak Z_{3}(12).
 \label{eq:S3(-4/3)}
\end{gather}
Here in \eqref{eq:S4(2)} and \eqref{eq:S4(4)}, the Dirichlet beta value
\[
\beta(4)\colonequals \sum_{n=0}^\infty\frac{(-1)^n}{(2n+1)^4}
\]
 is yet another generalization of Catalan's constant
 \[
 G\colonequals \sum_{n=0}^\infty\frac{(-1)^n}{(2n+1)^2};
  \]
  in \eqref{eq:S4(2)}, \eqref{eq:S4(3)}, \eqref{eq:S3(-4)}, \eqref{eq:S4(-4)}, \eqref{eq:S3(-1/2)}, \eqref{eq:S4(-1/2)}, and \eqref{eq:S3(-4/3)}, we have the Dirichlet $L$-values
\begin{gather*}
 L_{8,2}(3)\colonequals\sum_{n=0}^\infty\frac{(-1)^{n(n+1)/2}}{(2n+1)^3} , \\
 L_{8,4}(4)\colonequals\sum_{n=0}^{\infty} (-1)^n\left[ \frac{1}{(4n+1)^4} +\frac{1}{(4n+3)^4}\right], \\
  L_{3,2}(4)\colonequals\sum_{n=0}^{\infty} \left[ \frac{1}{(3n+1)^4} -\frac{1}{(3n+2)^4}\right] , \\
L_{12,4}(3)\colonequals \sum _{n=0}^{\infty } \left[\frac{1}{(12 n+1)^3}-\frac{1}{(12 n+5)^3}-\frac{1}{(12 n+7)^3}+\frac{1}{(12 n+11)^3}\right].
\end{gather*}

For each positive integer $k$ greater than $1$, it is also worth mentioning that the convergent series $ \mathscr S_k(z)$ in \eqref{eq:genChenSk} can be written as a generalized hypergeometric function\vspace{3mm}
\begin{align*}
\mathscr{S}_k(z)={_{k+1}}F_{k}\!\left( \left. \begin{array}{@{}c@{}}\smash[t]{
\overset{k-1}{\overbrace{\tfrac{1}{2},\dots,\tfrac{1}{2}}}},1,1 \\[4pt]
\smash[b]{\underset{k}{\underbrace{\tfrac{3}{2},\dots,\tfrac{3}{2}}}}
\end{array}\right\vert \frac{z}{4}\right),
\end{align*}

\vspace{3mm}

\noindent
which does not reduce to elementary expressions. (For $ k=2$, using {\tt FunctionExpand} in {\tt Mathematica} v14.0, one can check that
\begin{align*}
&\frac{1-w^2 }{w}{_{3}}F_{2}\Biggl( \begin{array}{@{}c@{}}\frac{1}{2},1,1 \\[4pt]
\frac{3}{2},\frac{3}{2} \\
\end{array}\Bigg\vert -\frac{1}{4} \biggl(\frac{1-w^2}{w}\biggr)^2\Biggr)
\\
&\qquad{}
=-2 [\Li_2(w)-\Li_2(-w)]-2 \log (w) \log \biggl(\frac{1-w}{1+w}\biggr)+\frac{\pi ^2}{2}
\end{align*}
holds when $0<w<1$, but reductions for larger integers $k$ are not automated yet.) Our results in Theorems \ref{thm:Li2Li3} and \ref{thm:ChenCMZV} also admit natural extensions to analytic continuations of these ${_{k+1}}F_{k}$ functions, outside the domain of convergence for the infinite series $ \mathscr S_k(z)$.

\section{Evaluations of Chen's series}\label{sec:Chen}

To evaluate \eqref{firstopen} and \eqref{secondopen}, we first convert them into integrals over polylogarithmic expressions, and then compute their integral representations
 by the function {\tt MZIntegrate}
 in Au's {\tt MultipleZetaValues} package \cite{Au2022a}.

\begin{proof}[Proof of Theorem \ref{thm:Chen1}]
 In view of the beta integral identity
 \begin{align*}
 \int_{0}^{1} \bigg( \frac{x}{1 + x^2} \bigg)^{2n+1} \frac{\D x}{1+x^2}
 = \frac{1}{ 4 (2 n+1) \binom{2 n}{n}}
\end{align*}
 together with the series bisection yielding
 \[ \sum_{n=0}^{\infty} \frac{z^{2n+1}}{(2n+1)^2}
 = \frac{\Li_2(z)-\Li_2(-z)}{2} , \]
 we may recast Chen's series \eqref{firstopen} into
 \begin{align}
 2\int_0^1 \frac{ \Li_2\bigl(\frac{x}{1+x^2}\bigr)-\Li_2\bigl(-\frac{x}{1+x^2}\bigr)}{1+x^2} \,\D x, \label{eq:firstopen_int}
 \end{align}
 by reversing the order of integration and infinite summation (which is justified by the dominated convergence theorem).
 Feeding
\begin{verbatim}
MZIntegrate[(2*(PolyLog[2, x/(1 + x^2)]
 - PolyLog[2, -x/(1 + x^2)]))/(1 + x^2), {x, 0, 1}]
\end{verbatim}
to {\tt Mathematica} after loading Au's {\tt MultipleZetaValues} package \cite{Au2022a}, we receive an output that is equivalent to the desired evaluation in \eqref{eq:Chen1sln}.
\end{proof}

\begin{proof}[Proof of Theorem \ref{thm:Chen2}]As a variation on \eqref{eq:firstopen_int}, we may represent
 \eqref{secondopen} by \begin{align}
\frac{2}{{\rm i}}\int_0^1 \frac{ \Li_2\bigl(\frac{{\rm i}x}{1+x^2}\bigr)-\Li_2\bigl(-\frac{{\rm i}x}{1+x^2}\bigr)}{1+x^2}\, \D x.\label{eq:secondopen_int}
\end{align}Unfortunately, the function {\tt MZIntegrate} in the current version (v1.2.0) of Au's {\tt Multiple\-Zeta\-Values} package \cite{Au2022a} does not automatically evaluate \eqref{eq:secondopen_int} in terms of CMZVs.

Nevertheless, we may reincarnate \eqref{eq:secondopen_int} into a form that is amenable to {\tt MZIntegrate}, by considering symmetries and deforming contours. Concretely speaking, the inversion $ x\mapsto\frac1x$ and reflection $x\mapsto -x$ symmetries allow us to identify \eqref{eq:secondopen_int} with
\begin{align}
&\frac{1}{{\rm i}}\int_0^\infty\ \frac{ \Li_2\bigl(\frac{{\rm i}x}{1+x^2}\bigr)-\Li_2\bigl(-\frac{{\rm i}x}{1+x^2}\bigr)}{1+x^2} \,\D x
=\frac{1}{\pi}\int_{-\infty}^\infty\ \frac{ \Li_2\bigl(\frac{{\rm i}x}{1+x^2}\bigr)-\Li_2\bigl(-\frac{{\rm i}x}{1+x^2}\bigr)}{1+x^2} \log\frac{x}{{\rm i}}\,\D x.
\tag{\ref{eq:secondopen_int}$'$}
\end{align}
Subsequently, we may close the contour upwards while exploiting a jump relation
\[
\Li_2(\xi+{\rm i}0^+)-\Li_2(\xi-{\rm i}0^+)=2\pi {\rm i}\log \xi
\]
 for $ \xi\in(1,\infty)$, which yields
\begin{align}
2{\rm i}\int_{{\rm i}/\phi}^{\rm i}\frac{\log\bigl( -\frac{{\rm i}z}{1+z^{2}} \bigr)}{1+z^{2}}\log\frac{z}{{\rm i}}\D z+2{\rm i}\int_{{\rm i}}^{{\rm i}\phi}\frac{\log\bigl( \frac{{\rm i}z}{1+z^{2}} \bigr)}{1+z^{2}}\log\frac{z}{{\rm i}}\,\D z
\tag{\ref{eq:secondopen_int}$''$}\label{eq:secondopen_int''}
\end{align}
for $ \phi\colonequals \frac{\sqrt{5}+1}{2}$. Now, throwing the last two integrals as
\begin{verbatim}
FullSimplify[MZIntegrate[2*I*(Log[(-I)*(z/(1 +
z^2))]/(1 + z^2))*Log[z/I]*D[I*((Sqrt[5] - 1)/2)*(1 - t) +
I*t, t] //.z -> I*((Sqrt[5] - 1)/2)*(1 - t) +
I*t, {t, 0, 1}] + MZIntegrate[2*I*(Log[I*(z/(1 + z^2))]/(1 +
z^2))*Log[z/I]*D[I*((Sqrt[5] + 1)/2)*t + I*(1 -
t), t] //. z -> I*((Sqrt[5] + 1)/2)*t + I*(1 -
t), {t, 0, 1}]]
\end{verbatim}
 into {\tt Mathematica}, we may confirm \eqref{eq:Chen2sln}.
\end{proof}

\section{Further generalizations of Chen's series\label{sec:genChen}}

\begin{proof}[Proof of Theorem \ref{thm:Li2Li3}]
By a natural extension of \eqref{eq:secondopen_int''}, we may equate the infinite series in~\eqref{eq:genChenLi2Li3} with
\begin{align}
&{}2{\rm i}\int_{{\rm i}w}^{\rm i}\frac{\log\bigl( \frac{1-w^{2}}{{\rm i}w}\frac{z}{1+z^{2}} \bigr)}{1+z^{2}}\log\frac{z}{{\rm i}}\,\D z+2{\rm i}\int_{{\rm i}}^{{\rm i}/w}\frac{\log\bigl(-\frac{1-w^{2}}{{\rm i}w} \frac{z}{1+z^{2}} \bigr)}{1+z^{2}}\log\frac{z}{{\rm i}}\,\D z.
\label{eq:genChen_intC}
\end{align}
Integrals of this type can be handled by Panzer's {\tt HyperInt} package \cite{Panzer2015} for {\tt Maple}. For instance, one may type
\begin{verbatim}
f := (w, z) -> 2*I*log(-I*(1 - w^2)*z/(w*(1 + z^2)))
*log(-I*z)/(1 + z^2);
g := (w, z) -> 2*I*log((1 - w^2)*z*I/(w*(1 + z^2)))
*log(-I*z)/(1 + z^2);
fibrationBasis(fibrationBasis(fibrationBasis
(hyperInt(f(w, z), [z = 0 .. I]), [w])
- fibrationBasis(hyperInt(f(w, z), [z = 0 .. I*w]), [w])
- fibrationBasis(hyperInt(g(w, z), [z = 0 .. I]), [w])
+ fibrationBasis(hyperInt(g(w, z), [z = 0 .. I/w]), [w])
+ (2*polylog(3, (1 + w)/2) - 2*polylog(3, (1 - w)/2)
- 2*polylog(3, 1/2*(1 + 1/w)) + 2*polylog(3, 1/2*(1 - 1/w))
- (polylog(2, (1 + w)/2) - polylog(2, (1 - w)/2)
+ polylog(2, (1 + 1/w)/2) - polylog(2, (1 - 1/w)/2))*log(w)
- Pi*delta[w]*log((1 + w)/2)*log((1 + 1/w)/2)*I), [w]));
\end{verbatim} and check that the output is zero. This verifies \eqref{eq:genChenLi2Li3} under the condition that $\I w>0 $, since the symbol {\tt delta[w]} in the last line of our code represents $ \I w/|{\I w}|$ when $\I w\neq0
$ \cite[formula~(3.6)]{Panzer2015}. After this, the case where $\I w=0$ can be determined by a limit procedure and a scrupulous analysis for the values of polylogarithms on their branch cuts. \end{proof}

\begin{Remark}One may recover the evaluations in Theorems \ref{thm:Chen1} and \ref{thm:Chen2} by setting $ w={\rm e}^{\pi {\rm i}/6}$ and ${w=\frac{\sqrt{5}-1}{2} }$ in~\eqref{eq:genChenLi2Li3} before invoking the function {\tt MZExpand} in Au's {\tt MultipleZetaValues} package~\cite{Au2022a}.\end{Remark}

Now, so long as $ \I({\rm i}w)>0$ and $ \I({\rm i}/w)>0$, we may upgrade our derivations for \eqref{eq:genChen_intC} into
\begin{align}
\sum _{n=0}^{\infty } \frac{(-1)^n }{(2 n+1)^k \binom{2 n}{n}}\biggl(\frac{1-w^2}{w}\biggr)^{2 n+1}
&{}=2{\rm i}\int_0^1 \frac{ \Li_{k-1}\bigl(\frac{1-w^2}{{\rm i}w}\frac{x}{1+x^2}\bigr)-\Li_{k-1}\bigl(-\frac{1-w^2}{{\rm i}w}\frac{x}{1+x^2}\bigr)}{1+x^2} \,\D x\nonumber\\
&{}=\frac{2{\rm i}}{(k-2)!}\int_{{\rm i}w}^{\rm i}\frac{\log^{k-2}\bigl( \frac{1-w^{2}}{{\rm i}w}\frac{z}{1+z^{2}} \bigr)}{1+z^{2}}\log\frac{z}{{\rm i}}\,\D z\nonumber\\
&\quad{}+\frac{2{\rm i}}{(k-2)!}\int_{{\rm i}}^{{\rm i}/w}\frac{\log^{k-2}\bigl(-\frac{1-w^{2}}{{\rm i}w} \frac{z}{1+z^{2}} \bigr)}{1+z^{2}}\log\frac{z}{{\rm i}}\,\D z,
\label{eq:genChen'_intC}
\end{align}
after recalling a jump relation $ \Li_s(\xi+{\rm i}0^+)-\Li_s(\xi-{\rm i}0^+)=\frac{2\pi {\rm i}}{(s-1)!}\log^{s -
 1} \xi$ for $ \xi\in(1,\infty)$.

To evaluate the last two contour integrals in \eqref{eq:genChen'_intC}, we need generalized polylogarithms (GPLs), which are defined recursively by an integral along a straight line segment
\begin{align}
G(\alpha_{1},\dots,\alpha_n;z)\colonequals\int_0^z\frac{G(\alpha_2,\dots,\alpha_n;x)\,\D x}{x-\alpha_1}\label{eq:GPL_rec}
\end{align}
for $ |\alpha_1|+\dots+|\alpha_n|\neq0$, with the boundary conditions that
\begin{align}
G(\underset{m }{\underbrace{0,\dots,0 }};z)\colonequals\frac{\log^mz}{m!},\qquad G(-\!\!-;z)\colonequals1.\label{eq:GPL_log_bd}
\end{align}
 One can convert GPLs to (analytic continuations of) MPLs [cf. \eqref{eq:CMZV_Li_defn}] via the following equation:
\begin{gather}
\vphantom{\underset{1}{\underbrace{0}}}G(\smash[b]{\underset{a_1-1 }{\underbrace{0,\dots,0 }}},\widetilde \alpha_1,\smash[b]{\underset{a_2-1 }{\underbrace{0,\dots,0 }}},\widetilde \alpha_2,\dots,\smash[b]{\underset{a_n-1 }{\underbrace{0,\dots,0 }}},\widetilde \alpha_n;z)
=(-1)^n\Li_{a_1,\dots,a_n}\biggl( \frac{z}{\widetilde \alpha_1} ,\frac{\widetilde \alpha_1}{\widetilde \alpha_2},\dots, \frac{\widetilde \alpha_{n-1}}{\widetilde \alpha_n}\biggr),\label{eq:GPL_MPL}\!
\end{gather}
where $ \prod_{j=1}^n\widetilde \alpha_j\neq0$. The corresponding integral relations date back to the work of E.E.~Kummer \cite{Kummer1840p1,Kummer1840p2,Kummer1840p3}, while the GPL-MPL\ conversion \eqref{eq:GPL_MPL} can be found in \cite[Section~4.2]{BorweinBradleyBroadhurstLisonek2001}. This allows us to reformulate \eqref{eq:Zk(N)_defn} into \begin{align}
\mathfrak Z_{k}(N)\colonequals\Span_{\mathbb Q}\left\{G(z_1,\dots,z_k;z)\left|\begin{matrix}z_1^N,\dots,z_{k}^N\in\{0,1\},\\z_1\neq1,\, z_{k}\neq0,\, z^{N}=1\end{matrix}\right.\right\}.\tag{\ref{eq:Zk(N)_defn}$'$}\label{eq:Zk(N)_defn'}
\end{align}

In view of \eqref{eq:GPL_rec} and \eqref{eq:GPL_log_bd}, we have $ G(\pm1;z/{\rm i})=\log(1\pm {\rm i}z)$ and $ G(0;z/{\rm i})=\log\frac{z}{{\rm i}}$. The\ function \smash{$ \log^{k-2}\bigl( \pm\frac{z/{\rm i}}{1+z^{2}} \bigr)\log\frac{z}{{\rm i}}$} is a $ \mathbb Q$-linear combination of GPLs in the form of $ G(\alpha_1,\dots,\alpha_{k-1};z/{\rm i})$, where $\alpha_1,\dots,\alpha_{j}\in\{-1,0,1\} $, as can be seen by repeated invocations of a shuffle product \cite{Blumlein2004,Hoffman2000}
\begin{align*}
G(\alpha;t)G(\beta_1,\dots,\beta_r;t)
&{}=G(\alpha,\beta_1,\dots,\beta_r;t)+\sum_{j=1}^{r-1}G(\beta_{1},\dots,\beta_j,\alpha,\beta_{j+1},\dots ,\beta_r;t)\\
&\quad{}+G(\beta_1,\dots,\beta_r,\alpha;t).
\end{align*}
Therefore, the GPL recursion in \eqref{eq:GPL_rec} tells us that
\begin{align}
&{\rm i}\int_{{\rm i}w}^{\rm i}\frac{\log^{k-2}\bigl( \frac{z/{\rm i}}{1+z^{2}} \bigr)}{1+z^{2}}\log\frac{z}{{\rm i}}\,\D z\in{}\mathfrak Z_{k}(2)+\mathfrak H_k^{(w)}(2),\label{eq:ZkHw}\\
&{\rm i}\int_{{\rm i}}^{{\rm i}/w}\frac{\log^{k-2}\bigl(- \frac{z/{\rm i}}{1+z^{2}} \bigr)}{1+z^{2}}\log\frac{z}{{\rm i}}\,\D z\in{}\mathfrak Z_{k}(2)+\mathfrak H_k^{(1/w)}(2),\label{eq:ZkHw'}
\end{align}
where
\begin{align}
&\mathfrak H_k^{(z)}(2)\colonequals{}\Span_{\mathbb Q}\left \{(\pi {\rm i})^{r-\ell}G(\alpha_1,\dots,\alpha_\ell;z)\left|\begin{matrix}\alpha_1^{2},\dots,\alpha _\ell^{2}\in\{0,1\},\\0\leq\ell\leq r,\, \alpha_{1}\neq z\end{matrix}\right.\right\}\nonumber\\
&\qquad \xlongequal{\text{\cite[p.\ 24]{DuhrDulat2019}}}{}\Span_{\mathbb Q}\left \{(\pi {\rm i})^{r-\ell}G(\alpha_1,\dots,\alpha_m;z)\log^{\ell-m}z\left|\begin{matrix}\alpha_1^{2},\dots,\alpha _m^{2}\in\{0,1\},\\0\leq m\leq\ell\leq r,\\\alpha_{1}\neq z,\,\alpha_m\neq0\end{matrix}\right.\right\}
\label{eq:HPL_space}\end{align}
is a $\mathbb Q $-vector space of \textit{hyperlogarithms} \cite{Maitre2005,Maitre2012}.
\begin{proof}[Proof of Theorem \ref{thm:ChenCMZV}]
 (a) For the left-hand side of \eqref{eq:genChenZa}, there are four
possible choices of $w$, namely, $ \pm {\rm e}^{2m\pi {\rm i}/N}$ and $ \pm {\rm e}^{-2m\pi {\rm i}/N}$. Without loss of generality, we may assume that $\I\bigl({\rm i}{\rm e}^{2m\pi {\rm i}/N}\bigr)>0$ and $ \I\bigl({\rm i}{\rm e}^{-2m\pi {\rm i}/N}\bigr)>0$, so that $w={\rm e}^{2m\pi {\rm i}/N} $ is applicable to \eqref{eq:genChen'_intC}. Consequently, the right-hand sides of both \eqref{eq:ZkHw} and \eqref{eq:ZkHw'} are subsets of $\mathfrak Z_{k}(\lcm(2,N)) $, as one can check by comparing~\eqref{eq:Zk(N)_defn'} and \eqref{eq:HPL_space} against the fact that $ \pi {\rm i}\in\mathfrak Z_1(N)$ for $ N-2\in \mathbb Z_{>0}$ \cite[Lemma~4.1]{Au2022a}. In addition, the relation
\begin{align*}
&{}\log^k\bigg(\frac{1-w^{2}}{w}\bigg)=[G(-1;w)-G(0;w)+G(1;w)]^{k}\in{}\mathfrak H^{(w)}_k(2)\subseteq \mathfrak Z_{k}(\lcm(2,N))
\end{align*}
also holds for $ w={\rm e}^{2m\pi {\rm i}/N} $. Therefore, the right-hand side of \eqref{eq:genChenZa} is a result of Goncharov's filtration $ \mathfrak Z_j(M)\mathfrak
Z_{ k}(M)\subseteq \mathfrak
Z_{j+ k}(M)$ \cite[Section~1.2]{Goncharov1998} for $M=\lcm(2,N)$.

(b) For $ N-2\in \mathbb Z_{>0}$ and $ w\neq0$, we have an embedding $ \mathfrak H_k^{(w)}(2)\subseteq\mathfrak Z_k(N)$ if $ \log w\in \mathfrak Z_k(N)$ and \begin{verbatim}IterIntDoableQ[{0, 1/w, -1/w}]\end{verbatim} returns a positive divisor of $N$ in Au's {\tt MultipleZetaValues} package \cite{Au2022a}. Particular cases of such embeddings include{\allowdisplaybreaks
\begin{alignat*}{3}
  &\mathfrak H_k^{(2)}(2)\subseteq{} \mathfrak Z_k(6),\qquad&& \mathfrak H_k^{(1/2)}(2)\subseteq{} \mathfrak Z_k(6), &\\
 & \mathfrak H_k^{(\sqrt{2}+1)}(2)\subseteq{} \mathfrak Z_k(8),\qquad&& \mathfrak H_k^{(\sqrt{2}-1)}(2)\subseteq{} \mathfrak Z_k(8),&\\
  &\mathfrak H_k^{(\sqrt{2})}(2)\subseteq{} \mathfrak Z_k(8),\qquad&& \mathfrak H_k^{(1/\sqrt{2})}(2)\subseteq{} \mathfrak Z_k(8), &\\
 & \mathfrak H_k^{(\phi)}(2)\subseteq{} \mathfrak Z_k(10),\qquad&& \mathfrak H_k^{(1/\phi)}(2)\subseteq{} \mathfrak Z_k(10), &\\
  &\mathfrak H_k^{(\sqrt{5})}(2)\subseteq{} \mathfrak Z_k(10),\qquad&& \mathfrak H_k^{(1/\sqrt{5})}(2)\subseteq{} \mathfrak Z_k(10),&\\
 & \mathfrak H_k^{(\sqrt{3})}(2)\subseteq{} \mathfrak Z_k(12),\qquad&& \mathfrak H_k^{(1/\sqrt{3})}(2)\subseteq{} \mathfrak Z_k(12).&
\end{alignat*}
}Meanwhile, we also have
\begin{align*}
\log\bigg(\frac{1-w^{2}}{w}\bigg)=\begin{cases}\log\frac{3}{2}\in\mathfrak Z_1(6) & \text{if }w=\frac{1}{2},\ \\
\log2\in\mathfrak Z_1(2)\subset\mathfrak Z_1(8) & \text{if }w=\sqrt{2}-1, \\\log\frac{1}{\sqrt{2}}\in\mathfrak Z_1(2)\subset\mathfrak Z_1(8)&\text{if }w=\frac{1}{\sqrt{2}},\\
0 & \text{if }w=\frac{1}{\phi} ,\\\log\frac{4}{\sqrt{5}}\in\mathfrak Z_1(10)& \text{if }w=\frac{1}{\sqrt{5}} ,\\\log\frac{2}{\sqrt{3}}\in\mathfrak Z_1(6)\subset\mathfrak Z_1(12)&\text{if }w=\frac{1}{\sqrt{3}}.
\end{cases}
\end{align*} Therefore, the right-hand sides of \eqref{eq:genChen'a}--\eqref{eq:genChen'd} all result from Goncharov's filtrations \cite[Section~1.2]{Goncharov1998} $ \mathfrak Z_j(N) \mathfrak
Z_{ k}(N)\subseteq \mathfrak Z_{j+ k}(N)$ for $ N\in\{6,$ $8,$ $10,$ $12\}$.
\end{proof}

\begin{Remark}Evaluating the integral representations in part (a) explicitly for $w={\rm e}^{\pi {\rm i}/4} $, ${w={\rm e}^{\pi {\rm i}/3}}$, and $w={\rm e}^{\pi {\rm i}/2}$ in Au's {\tt MultipleZetaValues} package \cite{Au2022a}, we get \eqref{eq:S3(2)}--\eqref{eq:S4(4)}. A~similar service on part (b) brings us \eqref{eq:S3(-9/4)}--\eqref{eq:S3(-4/3)}.
\end{Remark}\begin{Remark}To symbolically check any individual case among \eqref{eq:S3(2)}--\eqref{eq:S3(-4/3)}, one simply implements \eqref{eq:genChen'_intC} as
\begin{verbatim}((2*I)/(k - 2)!)*
(MZIntegrate[(Log[((1 - w^2)/(I*w))*(z/(1 + z^2))]^(k - 2)
/(1 + z^2))*Log[z/I]*D[I*w*(1 - t) + I*t, t]
//. z -> I*w*(1 - t) + I*t, {t, 0, 1}] +
MZIntegrate[(Log[(-((1 - w^2)/(I*w)))*(z/(1 + z^2))]^(k - 2)
/(1 + z^2))*Log[z/I]*D[I*(1 - t) + (I/w)*t, t]
//. z -> I*(1 - t) + (I/w)*t, {t, 0, 1}])\end{verbatim} in \texttt{Mathematica}, after choosing appropriate values for $k$ and $w$. For example, we need $k=3$ and $w={\rm e}^{\pi {\rm i}/4}$ for the analytic expression of $ \frac{1}{\sqrt{2}}\frac{1-w^2}{w}=-{\rm i}$ times \eqref{eq:S3(2)}. \end{Remark}

So far, we have limited the scope of this section to $ \mathscr{S}_k(z)$ where $k-1\in\mathbb Z_{>0}$. Here, we point out that both
\[
\mathscr{S}_1(z)=\frac{4}{\sqrt{z(4-z)}}\arcsin\frac{\sqrt{z}}{2}, \qquad
\mathscr{S}_0(z)=\frac{4 }{\sqrt{4-z} (4-z)}\left(\sqrt{4-z}+\sqrt{z} \arcsin\frac{\sqrt{z}}{2}\right)
\] are classical results, while the closed forms of $ \mathscr{S}_k(z)$ for negative integers $k$ can be deduced from the well-studied series of Ap\'ery--Lehmer type \cite{DysonFrankelGlasser2013,Glasser2012,Young2022}
\begin{align*}
 S_{\ell}(z)\colonequals \sum_{n=1}^\infty\frac{n^\ell z^n}{\binom{2n}{n}},
\end{align*}
where $ \ell\in\{1,2,\dots,|k|\}$.

\subsection*{Acknowledgements}
 J.M.C.~gratefully acknowledges support by a Killam Postdoctoral Fellowship from the Killam Trusts
 and wants to thank Karl Dilcher for useful feedback related to the sums investigated in this article. The research of Y.Z.~was supported in part by the Applied Mathematics Program within the Department of Energy
(DOE) Office of Advanced Scientific Computing Research (ASCR) as part of the Collaboratory on
Mathematics for Mesoscopic Modeling of Materials (CM4). The authors are thankful for the expert reviewer feedback that has led to numerous
 improvements to this paper.

\pdfbookmark[1]{References}{ref}
\LastPageEnding

\end{document}